\documentclass[11pt]{amsart}

\usepackage{amsmath,amssymb}
\usepackage{accents}
\usepackage{graphicx}
\usepackage{verbatim}

\newlength{\dhatheight}

\newtheorem{theorem}{Theorem}[section]

\newtheorem{proposition}[theorem]{Proposition}

\newtheorem{remark}[theorem]{Remark}
\newtheorem{corollary}[theorem]{Corollary}
\newtheorem{definition}[theorem]{Definition}

\newcommand{\N}{{\mathbb N}}
\newcommand{\Q}{{\mathbb Q}}
\newcommand{\R}{{\mathbb R}}
\newcommand{\Z}{{\mathbb Z}}

\newcommand{\ep}{\epsilon}
\newcommand{\dd}{\Delta}
\newcommand{\ra}{\rightarrow}
\newcommand{\ras}{{\stackrel{~~*}{\ra}}}
\newcommand{\ga}{\Gamma}

\newcommand{\pp}{{\mathcal{P}}}

\newcommand{\td}{d}

\newcommand{\tht}{\Theta}

\newcommand{\nff}{{normal filling}}
\newcommand{\edg}{{normal form diagram}}

\newcommand{\cc}{{\mathcal{N}}}

\newcommand\ct{{T}}

\newcommand{\stkg}{{stacking}}
\newcommand{\stkbl}{{stackable}}

\newcommand{\fstkbl}{{stackable}}
\newcommand{\astkg}{{algorithmic stacking}}
\newcommand{\astkbl}{{algorithmically stackable}}

\newcommand{\afstkbl}{{algorithmically stackable}}

\newcommand{\prs}{{\stackrel{~~p*}{\ra}}}

\newcommand{\mue}{\mu^e}

\newcommand{\dire}{{\vec E(X)}}  
\newcommand\ves{{\vec E_{r}}}  
\newcommand\dgd{{\vec E_d}}  

\newcommand{\hr}{\sff(\ves)}
\newcommand{\wa}{e_{w,a}}
\newcommand{\xa}{e_{x,a}}
\newcommand{\alg}{S_\ff}

\newcommand{\ega}{e_{g,a}}

\newcommand{\sff}{\phi}
\newcommand{\ff}{\Phi}
\newcommand{\ttt}{{\mathcal T}}
\newcommand{\vece}{{\vec E}}
\newcommand{\lbl}{\lambda}
\newcommand{\rep}{\rho}
\newcommand{\path}{\delta}
\newcommand{\init}{\iota}

\newcommand{\sredp}{reduction procedure}
\newcommand{\srecp}{stacking recursion} 

\begin{document}
\title[A uniform model for almost convexity and rewriting systems]
{A uniform model for almost convexity and rewriting systems}

\author[M.~Brittenham]{Mark Brittenham}
\address{Department of Mathematics\\
        University of Nebraska\\
         Lincoln NE 68588-0130, USA}
\email{mbrittenham2@math.unl.edu}

\author[S.~Hermiller]{Susan Hermiller}
\address{Department of Mathematics\\
        University of Nebraska\\
         Lincoln NE 68588-0130, USA}
\email{smh@math.unl.edu}


\begin{abstract}
We introduce a topological property for 
finitely generated groups
called \stkbl\ that
implies the existence of 
an inductive procedure for constructing
van Kampen diagrams with respect to a
particular finite presentation. 
We also define \astkbl\ groups, for which
this procedure is an algorithm.
This property gives a common
model for algorithms arising from both rewriting
systems and almost convexity for groups.
\end{abstract}

\maketitle


\section{Introduction}\label{sec:intro}


In geometric group theory, several properties of finitely
generated groups have been defined using a language of
normal forms together with geometric or combinatorial conditions
on the associated Cayley graph, most notably in the concepts of
combable groups and automatic groups in which the normal forms
satisfy a fellow traveler property. 
In this paper we use 
a set of normal forms together with another
topological property
on the Cayley graph of a finitely generated group, 
namely a notion
of ``flow'' toward the identity vertex,
to define a property 
which we call \stkbl.
  
More specifically, let $G$ be a group with a finite
inverse-closed generating set $A$, and   
let $\ga=\ga(G,A)$ be the associated
Cayley graph, with set $\vece$ of directed edges.  
For each $g \in G$ and $a \in A$,
let $\ega$ denote the directed edge in $\vece$
with initial vertex $g$, terminal vertex $ga$,
and label $a$.
Given a set $\cc \subset A^*$ of normal forms for $G$ over $A$,
write $y_g$ for the normal form of 
the element $g$ of $G$.
Note that whenever an equality of 
words $y_ga=y_{ga}$ or $y_g=y_{ga}a^{-1}$ holds, 
there is a van Kampen diagram for the word
$y_gay_{ga}^{-1}$ that contains no 2-cells;
in this case, we call the edge $\ega$ {\em degenerate}.
Let $\dgd=\vec E_{d,\cc}$ be the set of degenerate
edges, and let
$\ves =\vec E_{r,\cc}:= \vece \setminus \dgd$.

\begin{definition}\label{def:fstkbl}
A group $G$ is {\em \stkbl} with respect to a finite 
inverse-closed
generating set $A$ if there exist a set $\cc$ of normal forms
for $G$ over $A$ containing the empty word, a well-founded 
strict partial ordering $<$ on $\ves$,
and a constant $k$, such that
for every $g \in G$ and $a \in A$, 
there exists a 
path $p$
from $g$ to $ga$ in
$\ga$ 
of length at most $k$ 
satisfying the property that
whenever $e'$ is a directed edge in the path $p$,
either $e',\ega \in \ves$ and $e' < \ega$, or else
$e' \in \dgd$.
\end{definition}

In Section~\ref{sec:stackdef} we show 
in Proposition~\ref{prop:prefixclosed} that the set $\cc$
of normal forms for a \stkbl\ group is
closed under taking prefixes.  Thus 
$\cc$ determines 
a maximal tree in $\ga$, namely the
edges lying on paths that are labeled by words in $\cc$
and that start at the vertex labeled by the identity of $G$.
This leads to the following topological description 
of stackability.
Let $\vec P$ denote the set of directed
paths in $\ga$.  
For each $g \in G$ and $a \in A$,
we view the two directed edges $e_{g,a}$ and $e_{ga,a^{-1}}$  of $\ga$ to
have a single underlying undirected edge in $\ga$.
A {\em flow} function
associated to a maximal tree $\ttt$ in $\ga$ is a
function $\ff:\vece \ra \vec P$ 
satisfying the properties that: 
\begin{itemize}
\item[(F1)] For each edge $e \in \vece$,
the path $\ff(e)$ has the same initial and terminal
vertices as $e$.
\item[(F2d)] If the undirected edge underlying $e$ 
lies in the tree $\ttt$, then $\ff(e)=e$.
\item[(F2r)]  The transitive closure 
$<_\ff$ of the relation $<$ on
$\vec E$, defined by 
\begin{itemize}
\item[]
$e' < e$ whenever $e'$ lies on the path $\ff(e)$
and the undirected edges underlying both
$e$ and $e'$ do not lie in $\ttt$,
\end{itemize}
is a well-founded strict
partial ordering.
\end{itemize}
That is, the map $\ff$ fixes the edges lying in the tree $\ttt$
and describes a ``flow'' of the
non-tree edges toward the tree (or toward the basepoint);
starting from a non-tree edge and
iterating this function finitely many times results
in a path in the tree.
A flow function is {\em bounded} if there is
a constant $k$ such that for all $e \in \vec E$,
the path $\ff(e)$ has length at most $k$.

\smallskip

\noindent {\bf Corollary~\ref{cor:stkbldefs}.}  
{\em A group $G$ is stackable with respect to a
finite symmetric generating set $A$ if and
only if the Cayley graph
$\ga(G,A)$ admits a bounded flow function.
}

\smallskip

\noindent The two equivalent descriptions of stackability
in Corollary~\ref{cor:stkbldefs} are written to display 
connections to two other properties exploited later in the paper:  
Definition~\ref{def:fstkbl} closely resembles
 Definition~\ref{def:ac} of
almost convexity, and 
the bounded flow function 
is analogous to rewriting operations.

We show that every \stkbl\ group is
finitely presented
(in Proposition~\ref{prop:prefixclosed}) and admits
an inductive procedure which, upon input of a word 
in the generators that represents the identity of the group, 
constructs a van Kampen
diagram for that word over this presentation.
These van Kampen diagrams 
for \stkbl\ groups are constructed by building
up stacks of
van Kampen diagrams associated to
recursive edges, leading to the terminology
``\stkbl\ groups''. 
Letting $\rep:A^* \ra G$ be the canonical monoid
homomorphism and letting $\lbl:\vec P \ra A^*$ map
each path in $\ga$ to the word labeling its edges, 
we define a group $G$ to be {\em algorithmically stackable}
if the subset 
$$
S_\ff:=\{(w,a,\lbl(\ff(e_{\rep(w),a}))) \mid w \in A^*, a \in A\}
$$
of $A^* \times A \times A^*$
associated to a bounded flow function $\ff$ on $\ga$
is recursive.
This stronger property
guarantees that the inductive procedure for 
constructing van Kampen diagrams is an
algorithm.

%

\smallskip

\noindent{\bf Theorem~\ref{thm:solvwp}.} 
{\em 
If $G$ is \astkbl,
then $G$ has solvable word problem.
}

\smallskip

The \stkbl\ property provides a uniform
model for the 
procedures for
building van Kampen diagrams that arise
in both the example of groups with a finite complete
rewriting system and the example of almost convex groups
We discuss these and other examples of stackable
structures for groups in Section~\ref{sec:examples}.
To begin, in Section~\ref{subsec:bs} 
we give explicit details of a bounded
flow function for the Baumslag-Solitar group 
$BS(1,p)$ with $p \ge 3$.
In Section~\ref{sec:rs} we consider groups
that can be presented by rewriting systems, and
prove the following. 

\smallskip

\noindent{\bf Theorem~\ref{thm:crsrecit}.} {\em
A group admitting a finite complete
rewriting system is \astkbl.
}

\smallskip

In Section~\ref{subsec:f}, we consider Thompson's group $F$;
that is, the group of orientation-preserving piecewise linear
automorphisms of the unit interval for which all linear
slopes are powers of 2, and all breakpoints lie in the
the 2-adic numbers. 
In~\cite{chst}, Cleary, Hermiller, Stein, and Taback
show that Thompson's group 
$F$ is \fstkbl\ (although they do not use this
terminology, they build a \stkbl\ structure
in their construction of a 1-combing
for $F$), and their proof can be shown to give an \astkg.
We show in Section~\ref{subsec:f} that the 
set of normal forms associated to this \stkbl\ structre is a 
deterministic context-free language.
Thompson's group $F$ has been the focus of considerable
research in recent years, and yet the questions of
whether $F$ has a finite complete rewriting system  
or is automatic
are open (see the problem list at~\cite{thompsonpbms}).
Cleary and Taback~\cite{clearytaback}
have shown that Thompson's group $F$ is not almost convex
(in fact, Belk and Bux~\cite{belkbux}
have shown that $F$ is
not even minimally almost convex).
Thus $F$ is a potential example of an 
\astkbl\ group that has none of these 
other algorithmic and geometric properties.

A flow function can be
viewed as giving directions pointing from edges
toward the basepoint vertex labeled by the
identity $\ep$ of $G$, in the Cayley complex for 
the stacking presentation.
(For example, an illustration of this flow for the Baumslag-Solitar
group $BS(1,2)$ is given in Figure~\ref{fig:bs12}
in Section~\ref{subsec:bs}.)
That is, from any degenerate edge one can follow
the maximal tree $\ttt$ associated to $\cc$
to the next edge $e'$
along the unique simple path toward $\ep$, 
and from any recursive edge $e$, one can 
follow a 2-cell to a path containing
an edge $e'$ that is either degenerate,
else is recursive and satisfies 
$e' <_{\ff} e$.  In both cases, we view $e'$ 
as ``closer'' than $e$ to the basepoint.
A natural special case to consider occurs
when this notion of ``closer'' coincides 
with the path metric $d_X$ on the Cayley graph $X^1=\ga$.  
That is, define the function $\alpha:\vece \ra \Q$
by setting
$\alpha(e):= \frac{1}{2}(d_X(\ep,a)+d_X(\ep,b))$
for each edge $e \in \dire$ 
with endpoints $a$ and $b$, 
so that $\alpha$ measures the average distance from
a point of $e$ to the origin.
\begin{definition}\label{def:geostk}
A group $G$ is {\em geodesically \stkbl} if $G$ has a
finite symmetric generating set $A$ with a 
stackable structure over a normal form set $\cc$ and
an associated bounded flow
function $\ff$ 
such that all of the elements of $\cc$ label geodesic
paths in $\ga(G,A)$, and whenever $e',e \in \ves$ with
$e' <_{\ff} e$, then $\alpha(e')<\alpha(e)$.
\end{definition}
In Section~\ref{subsec:ac} we show that
this property is equivalent to
Cannon's almost convexity property~\cite{cannon} 
(see Definition~\ref{def:ac}).
The proof yields the somewhat unexpected result
that any geodesically stackable structure can be
replaced by another that is both algorithmic and
based upon the shortlex normal forms.

\smallskip

\noindent{\bf Theorem~\ref{thm:aceti}.}  {\em
Let $G$ be a group with finite symmetric generating set $A$.
The following are equivalent:
\begin{enumerate}
\item The pair $(G,A)$ is almost convex.
\item The pair $(G,A)$ is geodesically \stkbl.
\item The pair $(G,A)$ is geodesically \afstkbl\ with 
respect to shortlex normal forms.
\end{enumerate}
}

\smallskip

\noindent The properties in Theorem~\ref{thm:aceti}
are satisfied by all
word hyperbolic groups and  
cocompact discrete groups of isometries of Euclidean
space, with respect to every generating set~\cite{cannon}.
Hence Theorem~\ref{thm:aceti} shows that every
word hyperbolic group is \astkbl.




One of the motivations for the definition of automatic 
groups was to 
gain a better understanding of the fundamental
groups of 3-manifolds, in particular
to find practical methods for computing in these groups.
However, the fundamental group of a 3-manifold 
is automatic if and only if its JSJ decomposition does not 
contain manifolds with a uniform Nil or Sol 
geometry~\cite[Theorem~12.4.7]{echlpt}.
In contrast,~\cite{hs} Hermiller and Shapiro have shown that
the fundamental group of every closed 3-manifold with a
uniform geometry other than hyperbolic must have a
finite complete rewriting system, and so combining this result
with Theorems~\ref{thm:crsrecit} and~\ref{thm:aceti}
yields the following.

\smallskip

\begin{corollary} 
If $G$ is the fundamental group of a closed 3-manifold with
a uniform geometry, then $G$ is \astkbl. 
\end{corollary}

\smallskip 

The algorithmically stackable property also
allows a wider range of Dehn (or isoperimetric) functions
than those for automatic or 
combable groups, whose
Dehn functions are at most 
quadratic~\cite{echlpt} or exponential (shown
by Gersten; see, for 
example,~\cite{riley}),
respectively. 
In particular, the iterated Baumslag-Solitar group 
$$G_k=\langle a_0,a_1,...,a_k \mid a_i^{a_{i+1}}=a_i^2; 
0 \le i \le k-1\rangle$$
admits a finite complete rewriting
system for each $k \ge 1$ (first described by Gersten;
see \cite{hmeiermeastame} for details),
and so Theorem~\ref{thm:crsrecit}
shows that $G_k$ is \astkbl.
Gersten~\cite[Section~6]{gerstenexpid} showed that 
the Dehn function for $G_k$
grows at least as fast as the function
$$
n \mapsto \underbrace{2^{2^{.^{.^{.^{2^n}}}}}}_{\hbox{k times}}~.
$$
Hence the class of \astkbl\ groups includes groups
whose Dehn functions are 
towers of exponentials.



On the other hand, in~\cite{britherm}, the present authors show 
that stackable groups are tame combable, as defined
by Mihalik and Tschantz~\cite{mihaliktschantz}.
Tschantz~\cite{tschantz} has conjectured that there
exists a finitely presented group that is not
tame combable. Such a group would also not admit the
stackable property with respect to any finite symmetric generating
set.



\section{Notation} \label{sec:notation}


Throughout this paper, let $G$ be a group
with a finite {\em symmetric} generating set; that
is, such that the generating set $A$ is closed under inversion.
Throughout the paper we assume that no element of $A$
represents the identity element of $G$.

Let $\rep:A^* \ra G$ be the canonical monoid homomorphism.
A set $\cc$ of {\em normal forms} for $G$ over $A$ is a 
subset of 
$A^*$ such that the restriction of the
map $\rep$ to $\cc$ is a bijection.
As in Section~\ref{sec:intro}, the symbol $y_g$ denotes
the normal form for $g \in G$.  By slight abuse
of notation, we use the symbol $y_w$ to denote the
normal form for $\rep(w)$ whenever $w \in A^*$.

Let $\epsilon$ denote the identity of $G$, and let
$1$ denote the empty word in $A^*$.
For a word $w \in A^*$, we write $w^{-1}$ for the 
formal inverse of $w$ in $A^*$.
For words $v,w \in A^*$, we write $v=w$ if $v$
and $w$ are the same word in $A^*$, and write $v=_G w$ if
$v$ and $w$ represent the same element of $G$; that is, 
if $\rep(v)=\rep(w)$.

Let $\ga$ be the Cayley graph of $G$ with
respect to $A$, with path metric $d$.
%
A word $w \in A^*$ is called {\em geodesic}
if $w$ labels a geodesic path in $\ga$.
Whenever $x \in A^*$ and $a \in A$,
we write $\xa$ to denote the directed edge $\ega$
where  $g=\rep(x)$ is the element 
of $G$ represented by $x$.
Define four maps

$\alpha:\vece \ra \Q$ by 
  $\alpha(\ega) := \frac{1}{2}(d(\ep,g)+d(\ep,ga))$,

$\lbl:\vec P \ra A^*$ by $\lbl(p):=$ the word labeling the path $p$,

$\init:\vece \ra G$ by $\init(\ega):=g$, and

$\path:G \times A^* \ra \vec P$ by $\path(g,w):=$ the
  path in $\ga$ starting at $g$ labeled by $w$.

Given a presentation 
$\pp = \langle A \mid R \rangle$ for $G$, the
presentation is {\em symmetrized} if
the generating set $A$ is symmetric
and the set $R$ of defining relations is closed under
inversion and cyclic conjugation.
Let $X$ be the Cayley 2-complex corresponding to this presentation,
whose 1-skeleton is $X^1=\ga$.
Let $E(X)$ denote the set of undirected edges
of $X$; we consider the two directed edges
$\ega$ and $e_{ga,a^{-1}}$ to have the same
underlying directed edge in $X$ between $g$ and $ga$ in $X$.

For an arbitrary word $w$ in $A^*$
that represents the
trivial element $\ep$ of $G$, there is a {\em van Kampen
diagram} $\dd$ for $w$ with respect to $\pp$.  
That is, $\dd$ is a finite,
planar, contractible combinatorial 2-complex with 
edges directed and
labeled by elements of $A$, satisfying the
properties that the boundary of 
$\dd$ is an edge path labeled by the
word $w$ starting at a basepoint 
vertex $*$ and
reading counterclockwise, and every 2-cell in $\dd$
has boundary labeled by an element of $R$.
For any van Kampen diagram
$\dd$ with basepoint $*$, let $\pi_\dd:\dd \ra X$
denote a cellular map such that $\pi_\dd(*)=\ep$ and
$\pi_\dd$ maps edges to edges preserving both
label and direction.

In general, there may be many different van 
Kampen diagrams for the word $w$.  Also,
we do not assume that van Kampen diagrams
in this paper are reduced; that is, we allow adjacent
2-cells in $\dd$ to be labeled by the same relator with
opposite orientations.

See for example~\cite{bridson} or~\cite{lyndonschupp} 
for an exposition of the theory of van Kampen diagrams.


\section{Procedures for constructing normal forms and 
van Kampen diagrams} \label{sec:stackdef}


The main goal of this section is to describe 
inductive procedures for finding normal
forms and for constructing van Kampen diagrams
for stackable groups.  
We begin with a discussion of the structure of
the normal forms
for a stackable group.



Let $G$ be a group that is \stkbl\ over a
symmetric generating set $A$, with \stkbl\ structure
$(\cc,<,k)$ from Definition~\ref{def:fstkbl}.  Then one can
define a function $\sff:\ves=\vec E_{r,\cc} \ra A^*$
by choosing, for 
each $\ega \in \ves$, a label $\sff(\ega) =a_1 \cdots a_n \in A^*$
of a directed path in $\ga$ satisfying the property that
$\sff(\ega) =_G a$,
$n \le k$, and 
either 
$e_{ga_1 \cdots a_{i-1},a_i} < \ega$ or
$e_{ga_1 \cdots a_{i-1},a_i} \in \dgd$ for each $i$.
(Note that although we have $\sff(\ega) =_G a$,  
the fact that $<$ is a strict partial ordering 
implies that the word $\sff(\ega)$ cannot be
the letter $a$.)
(See Figure~\ref{fig:stkgmap}.)
\begin{figure}
\begin{center}
\includegraphics[width=3.8in,height=1.0in]{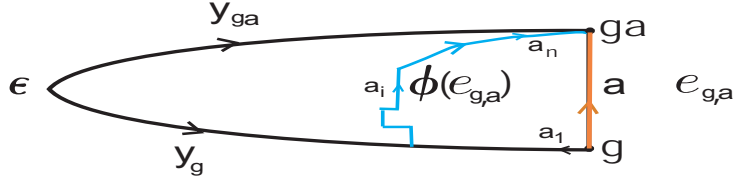}
\caption{The stacking map}\label{fig:stkgmap}
\end{center}
\end{figure}
This function is called a {\em \stkg\ map}.

Let $R_\sff$ be the closure of the set of words
$\{\sff(\ega) a^{-1} \mid g \in G, a \in A\}$ under inversion,
cyclic conjugation, and free reduction (except the
empty word); $R_\sff$ is
called the {\em stacking relation set}.

In the proof of the following proposition, we
give an inductive procedure which, upon input of any word
$w$ over the generators of a \stkbl\ group, will
output the normal form of the element $\rep(w)$ of the group $G$
represented by $w$.

\begin{proposition}\label{prop:prefixclosed}
Let $G$ be a stackable group over a generating set $A$.
Then $G$ is finitely presented by
 $\langle A \mid R_\sff\rangle$, 
where  $R_\sff$ is the \stkg\ relation set 
associated to a stacking map for $G$.
Moreover, the set $\cc$ of normal forms of a \stkbl\ structure is
closed under taking prefixes.
\end{proposition}

\begin{proof}
Let $\sff$ be a \stkg\ map associated to the
stackable structure on $G$ over $A$, with
normal form set $\cc$, ordering $<$, and constant $k$.
We begin by defining a relation $<_\sff$ on $\vece$ as follows.
Whenever $e',e$ are both in $\ves$ and $e'$ lies in
the path 
$\path(\init(e),\sff(e))$
in $\ga$ starting at the initial vertex of $e$
and labeled by $\sff(e)$
(and $e'$ is oriented 
in the same direction as this path), 
write $e' <_{\sff} e$.  Let
$<_{\sff}$ be the transitive closure of this relation.
Then $<_{\sff}$ is a subset of the well-founded 
strict partial ordering $<$ from 
Definition~\ref{def:fstkbl}, and so is also
a well-founded strict partial ordering.
Moreover, the constant bound $k$ on the lengths of
words $\sff(e)$
together with
K\"onig's Infinity Lemma imply that $<_{\sff}$ satisfies
the property that for each $e \in \ves$, there are
only finitely many $e'' \in \ves$ with $e'' <_{\sff} e$.

Using the stacking map, we
describe a \sredp\ for 
finding the normal form
for the group element associated to any word, 
by defining a rewriting operation on words over $A$, as follows.

Whenever a word $w \in A^*$ has a decomposition 
$w=xay$ such that $x,y \in A^*$, $a \in A$, and the directed
edge $\xa$ of $\ga$ 
lies in $\ves$,
then we rewrite $w \ra x \sff(\xa) y$.
Now for every directed edge $e'$ 
in the Cayley graph $\ga$ that lies along the path
$\path(\rep(x),\sff(x))$,
either $e'$ is
a degenerate edge in $\dgd$, or else $e' \in \ves$ and $e'<_{\sff} \xa$.
After rewriting a word
$w$ finitely many times $w \ra w_1 \ra \cdots \ra w_i$, 
any decomposition $w_i=x_ia_iy_i$ with $x_i,y_i \in A^*$,
$a_i \in A$, and $e_i:=e_{x_i,a_i} \in \ves$ satisfies
the property that $e_i <_{\sff} e$ where
$e=e_{x,a}$ for some decomposition $w=xay$ of the
original word $w$.
That is, each successive rewriting corresponds
to one of the finitely many edges that are less (with
respect to $<_{\sff}$) than
the finitely many edges in the path labeled $w$ in $\ga$
starting at the identity vertex.
Thus there can be at most finitely many rewritings 
$w \ra w_1 \ra \cdots \ra w_m=z$ 
until a word $z$ is obtained
which cannot be rewritten with this procedure.  
The final step of the \sredp\ is to freely reduce 
the word $z$, resulting in a word $w'$.

Now $w =_G w'$, and the word $w'$ (when input into 
this procedure) is not
rewritten with the \sredp,
since every prefix of $w'$ is equal in $G$ to a prefix of $z$.  
Write $w'=a_1 \cdots a_n$ with
each $a_i \in A$.  Then for all $1 \le i \le n$, the
edge $e_i:=e_{a_1 \cdots a_{i-1},a_i}$ of $\ga$ 
does not lie in $\ves$, and so must be in $\dgd$.  
In the case that $i=1$, this implies that
one of the equalities of words
 $y_\ep a_1 = y_{a_1}$ or $y_{a_1} a_1^{-1} = y_\ep$ must hold
(where $\ep$ denotes the identity of $G$).
Now from Definition~\ref{def:fstkbl} we have
that the normal form of the identity
is the empty word, i.e.~$y_\ep = 1$, and so the first equality
$a_1=y_{a_1}$ must hold.  Assume inductively that
$y_{a_1 \cdots a_i}=a_1 \cdots a_i$.  The inclusion
$e_{i+1} \in \dgd$ implies that either 
$a_1 \cdots a_i \cdot a_{i+1}=y_{a_1 \cdots a_{i+1}}$
or $y_{a_1 \cdots a_{i+1}} a_{i+1}^{-1} = a_1 \cdots a_{i}$.
However, the latter equality on words would imply
that the final letters on each side
are the same, i.e.~$a_{i+1}^{-1}=a_i$, which contradicts
the fact that $w'$ is freely reduced.  Hence we have
that $w'=y_{w'}=y_w$ is in normal form, and moreover every
prefix of $w'$ is also in normal form. 

This \sredp\ uses only relators
of the group lying in the stacking
relation set $R_\sff$ to reduce any word $w \in A^*$
to its normal form.  Hence $R_\sff$ is a set of defining
relators for $G$ over the generating set $A$.
Since the words in $R_\sff$ have 
length at most $k+1$, 
the set $R_\sff$ is also finite.
\end{proof}

We call $\langle A \mid R_\sff \rangle$
the {\em \stkg\ presentation}.
The prefix-closed set $\cc$ of normal
forms for a stackable group 
yields a maximal
tree $\ttt$ in the Cayley graph $\ga$,
namely the set of edges in the paths 
in $\ga$ 
starting at $\ep$ and labeled by the words in $\cc$.
In the following Corollary we show that
a stacking map yields a flow function associated
to this tree.

\begin{corollary}\label{cor:stkbldefs}
A group $G$ is stackable with respect to a
finite symmetric generating set $A$ if and
only if the Cayley graph
$\ga(G,A)$ admits a bounded flow function.
\end{corollary}

\begin{proof}
First suppose that $G$ is stackable over $A$,
and let $\sff:\ves \ra A^*$ be a stacking map.
From Proposition~\ref{prop:prefixclosed},
the set $\cc$ of normals forms from the 
stackable structure is prefix-closed; let
$\ttt$ be the maximal tree in the Cayley
graph $\ga=\ga(G,A)$ consisting of the
edges lying in paths starting at $\ep$ 
labeled by words in $\cc$. 
The set $\dgd$ of degenerate edges associated
to the normal form set $\cc$
is exactly
the set  of directed
edges lying in this tree, and the edges of
$\ves$ are the edges of $\ga$ that do not lie
in the tree $\ct$. 

Let $\ff:\vece \ra \vec P$ be the function given by
defining $\ff(\ega) :=  \ega$ whenever $\ega \in \dgd$
and defining $\ff(\ega) := \path(g,\sff(\ega))$, the directed path
in $\ga$ with initial vertex $g$
that is labeled by the word $\sff(\ega)$,
whenever $\ega \in \ves$.
That is, 
$\ff|_{\dgd} = id_{\dgd}$
and $\ff|_{\ves} = \path \circ (\init \times \sff)$.
Properties (F1), (F2d), and (F2r) 
of a flow function follow directly from the 
fact that $\sff$ is a stacking map.
The constant $k$ of the stackable structure
is also a bound for this flow function.

Conversely, given a bounded flow function
$\ff:\vece \ra \vec P$ associated to a
maximal tree $\ttt$ in the Cayley graph $\ga(G,A)$,
let $\cc$ be the set of normal forms for 
$G$ over $A$ consisting of the words labeling
paths starting at $\ep$
that are geodesic (i.e.~never backtrack) in $\ttt$.
Let $k$ be the constant bound on $\ff$,
and let $<$ be the restriction of the ordering
$<_\ff$ to $\ves$.  Then $(\cc,<,k)$ with
the stacking map $\sff:=\lbl \circ \ff|_{\ves}$
give a
stackable structure for $G$ over $A$.
\end{proof}

We note that the \sredp\  
described in the proof of Proposition~\ref{prop:prefixclosed},
for finding normal forms
for words, may not be an algorithm.   
To make this process 
algorithmic, we would need to be able to
recognize, given $x \in A^*$ and $a \in A$,
whether or not $\xa \in \ves$, and if so,
be able to find $\sff(\xa)$.  
That is, the 
set 
$$
\{(w,a,\sff(e_{w,a})) \mid e_{w,a} \in \ves\} \cup
\{(w,a,a) \mid e_{w,a} \in \dgd\}
$$
should be computable (i.e., decidable or recursive).
If we let $\ff$ be the flow function associated to
$\sff$ from Corollary~\ref{cor:stkbldefs}, 
then (using the notation from Section~\ref{sec:notation})
this set is the graph $S_\ff$ of the
function $A^* \times A \ra A^*$ given by
$(w,a) \mapsto \lbl(\ff(e_{w,a}))$.
In the case that $S_\ff$ is computable, given any $(w,a) \in A^* \times A$, 
by enumerating the words $z$ in $A^*$ and
checking in turn whether $(w,a,z) \in \alg$, we
can find $\lbl(\ff(e_{w,a}))$.
(Note that 
the set $\alg$ is computable if and only if
the set $\{(w,a,\sff(\wa)) \mid w \in A^*, a \in A, e_{w,a} \in \ves\}$
describing the graph of $\sff$ is computable.  
However, using the latter set in the
reduction algorithm has the drawback of requiring us to 
enumerate the finite (and hence enumerable) set $\hr$,
but we may not have an algorithm to find this
set from the \stkbl\ structure.)  
Hence we have shown the following.

\begin{theorem}\label{thm:solvwp}
If $G$ is \astkbl,
then $G$ has solvable word problem.
\end{theorem}

As with many other algorithmic classes of groups,
it is natural to discuss formal language theoretic
restrictions on the associated languages, and in
particular on the set of normal forms.
Computability of the set $\alg$ implies
that the set $\cc$ is computable as well 
(since any word $a_1 \cdots a_n \in A^*$ lies
in $\cc$ if and only if the word is freely
reduced and for each $1 \le i \le n$
the tuple $(a_1 \cdots a_{i-1},a_i,a_i)$ lies
in $\alg$).  Many of the examples we consider
in Section~\ref{sec:examples} will satisfy
stronger restrictions on the set $\cc$.

\medskip

Next we turn to a discussion of building 
van Kampen diagrams in \stkbl\ groups.
Before discussing the details of the inductive procedure 
for constructing these diagrams, we first reduce the 
set of diagrams required.

For a group $G$ with symmetrized
presentation $\pp=\langle A \mid R \rangle$,
a {\em filling} is a collection 
$\{\dd_w \mid w \in A^*, w=_G \ep\}$
of van Kampen diagrams for all words representing 
the trivial element.
Given a set 
$\cc=\{y_g \mid g \in G\} \subseteq A^*$ 
of normal forms for $G$,
a {\em \edg} is a van Kampen diagram for a word
of the form $y_gay_{ga}^{-1}$ where
$g \in G$ and $a$ in $A$.  We can associate
this {\edg} with the directed edge
of the Cayley complex $X$ labeled by $a$ with
initial vertex labeled by $g$.
A {\em \nff}~\label{defnff} for the pair $(G,\pp)$ 
consists of a set $\cc$ normal
forms for $G$ that are {\em simple words}
(i.e. words that label simple paths in the 
1-skeleton of the Cayley complex $X$)
including the empty word,
together with a collection
$\{\dd_e \mid e \in E(X)\}$ of {\edg}s, where
for each undirected edge $e$ in $X$, the \edg\ 
$\dd_e$ is associated to one of the two
possible directions of $e$.

Every \nff\ induces a filling, using the ``seashell''
(``cockleshell'' in \cite[Section~1.3]{riley})
method, illustrated in Figure~\ref{fig:seashell}, as follows.
\begin{figure}
\begin{center}
\includegraphics[width=3.6in,height=1.6in]{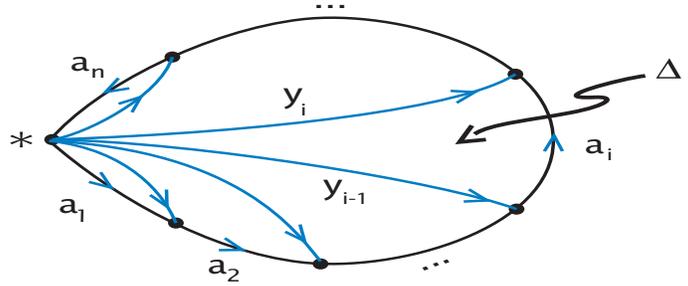}
\caption{Van Kampen diagram built with seashell procedure}\label{fig:seashell}
\end{center}
\end{figure}
Given a word $w=a_1 \cdots a_n$ representing
the identity of $G$, with each $a_i \in A$, then
for each $1 \le i \le n$, there is a \edg\ 
$\dd_i$ in the \nff\ that is associated
to the edge of $X$ with endpoints labeled by the group
elements represented by the words
$a_1 \cdots a_{i-1}$ and $a_1 \cdots a_i$.
Letting $y_i$ denote the normal form in $\cc$ representing 
$a_1 \cdots a_i$,
then the counterclockwise boundary of this diagram
is labeled by either 
$y_{i-1} a_i y_{i}^{-1}$
or $y_{i} a_i^{-1} y_{i-1}^{-1}$;
by replacing $\dd_i$ by its mirror image if necessary, we 
may take $\dd_i$ to have counterclockwise boundary 
word $x_i:=y_{i-1} a_i y_{i}^{-1}$.
We next iteratively build a van Kampen diagram 
$\dd_i'$ for the word 
$y_{\ep}a_1 \cdots a_i y_{i}^{-1}$,
beginning with $\dd_1':=\dd_1$.  For $1<i \le n$, 
the planar diagrams $\dd_{i-1}'$ and $\dd_{i}$
have boundary subpaths
sharing a common label $y_{i}$.
The fact that this word $y_i$ is simple, labeling
a simple path in $X$, 
implies that the paths in the van Kampen
diagrams $\dd_{i-1}',\dd_{i}$ labeled by $y_i$ 
must also be simple, since the path in $X$
is the image under the cellular maps
$\pi_{\dd_{i-1}'}$ and $\pi_{\dd_i}$ of these 
boundary paths.
Hence each of these boundary paths labeled $y_i$
is an embedding in the respective van Kampen diagram.
These paths are also oriented in the same direction,
and so the diagrams $\dd_{i-1}'$ and $\dd_{i}$ can be
glued, starting at their basepoints and
folding along these subpaths,
to construct the 
planar diagram $\dd_i'$.  
Performing these gluings consecutively for each
$i$ results in a van Kampen diagram $\dd_n'$ with
boundary label $y_{\ep}wy_{w}^{-1}$.
Note that we have allowed the
possibility that some of the boundary edges of $\dd_n'$
may not lie on the boundary of a 2-cell in $\dd_n'$;
some of the words $x_i$ may freely
reduce to the empty word, and the corresponding
van Kampen diagrams $\dd_i$ may have no 2-cells.
Note also that
the only simple word representing the identity of $G$
is the empty word; that is, $y_\ep =y_{w}=1$.  
Hence $\dd_n'$ is the
required van Kampen diagram for $w$.

Starting from a bounded flow function
$\ff$  for a stackable group
$G$ over a finite generating set $A$, 
we now describe the {\em \srecp}, which is
an inductive procedure for
constructing a filling for $G$ over the \stkg\ 
presentation $\pp=\langle A \mid R_\sff \rangle$ 
by building a \nff\ to which the seashell
method can be applied, as follows.  
Let $X$ be the Cayley complex of this presentation,
and let $\cc$ be the normal form set
obtained from the maximal tree $\ttt$ 
associated to $\ff$.
Since these normal forms label geodesics in
a tree, 
each element $w$
of $\cc$ must be a  
simple word, i.e.~labeling
a simple path in $X$.


We define a \edg\  corresponding to each
directed edge in $\vece=\ves \cup \dgd$ of the
Cayley graph as follows.
Let $e$ be an edge in $\vece$, oriented
from a vertex $g$ to a vertex $h$ and labeled by $a \in A$,
and let $w_e:=y_g a y_{h}^{-1}$.

In the case that $e$ lies in $\dgd$, the
word $w_e$ freely reduces to the empty
word.  Let $\dd_e$ be the van Kampen diagram for 
$w_e$ consisting of a line segment of edges,
with no 2-cells.
(See Figure~\ref{fig:degenstkg}.)
\begin{figure}
\begin{center}
\includegraphics[width=2.3in,height=0.7in]{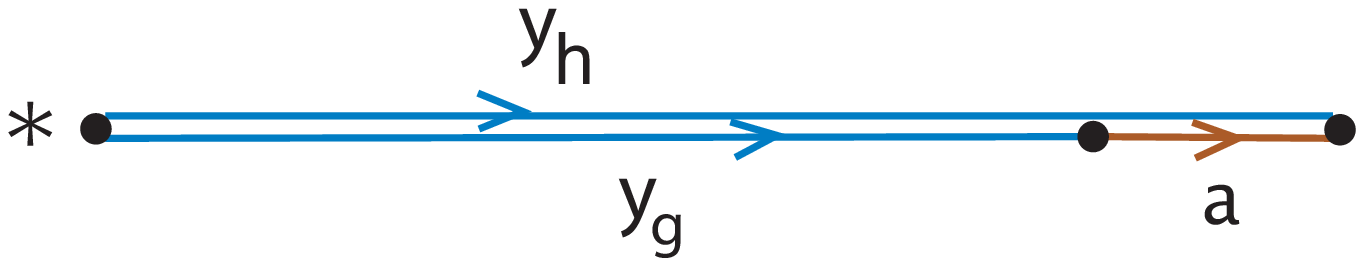}
\hspace{.2in}
\includegraphics[width=2.3in,height=0.7in]{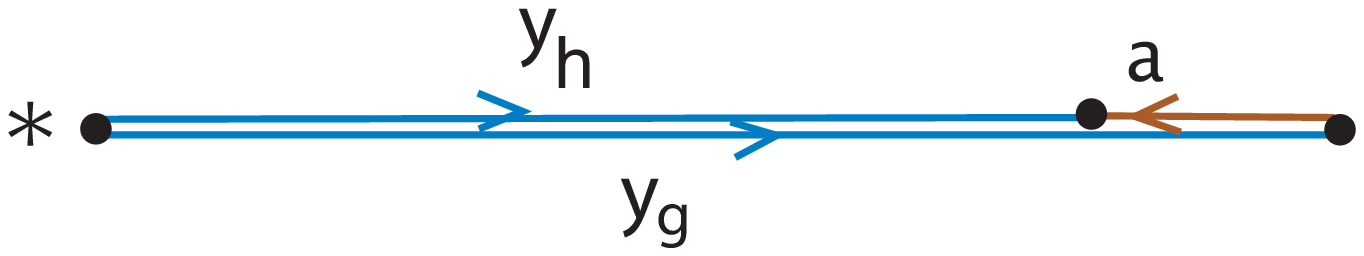}
\caption{Van Kampen diagram $\dd_e$ for degenerate edge $e$}\label{fig:degenstkg}
\end{center}
\end{figure}

In the case that $e \in \ves$, we
use Noetherian induction
to construct the \edg.
Write $\lbl(\ff(e))=a_1 \cdots a_n$ with each $a_i \in A^*$, and for each
$1 \le i \le n$, let $e_i$
be the edge in the Cayley graph from $ga_1 \cdots a_{i-1}$
to $ga_1 \cdots a_i$ labeled by $a_i$.  
For each $i$, either the directed edge $e_i$ is
in $\dgd$, or else $e_i \in \ves$ and 
$e_i <_{\ff} e$; in both cases we
have, by above or by Noetherian induction,
a van Kampen diagram $\dd_i:=\dd_{e_i}$ with boundary label
$y_{ga_1 \cdots a_{i-1}} a_i y_{ga_1 \cdots a_i}^{-1}$.
By using the ``seashell'' method, we successively 
glue the diagrams $\dd_{i-1}$, $\dd_i$ along their
common boundary words $y_{ga_1 \cdots a_{i-1}}$.
Since all of these gluings are along simple paths, 
this results in a planar van Kampen diagram $\dd_e'$ with boundary
word $y_g \lbl(\ff(e)) y_{h}^{-1}$.  
(Note that by our assumption that no generator represents
the identity, $\lbl(\ff(e))$ must contain at least one letter.)
Finally, glue a polygonal 2-cell with boundary
label given by the relator $\lbl(\ff(e))a^{-1}$ 
along the boundary subpath labeled $\lbl(\ff(e))$ in $\dd_e'$, in order to 
obtain the diagram $\dd_e$ with boundary 
word $w_e$.  Since in this step we have glued a disk onto $\dd_e'$
along an arc, the diagram $\dd_e$ is again planar, and is a \edg\ 
corresponding to $e$.
(See Figure~\ref{fig:genericstacking}.)
\begin{figure}
\begin{center}
\includegraphics[width=3.6in,height=1.6in]{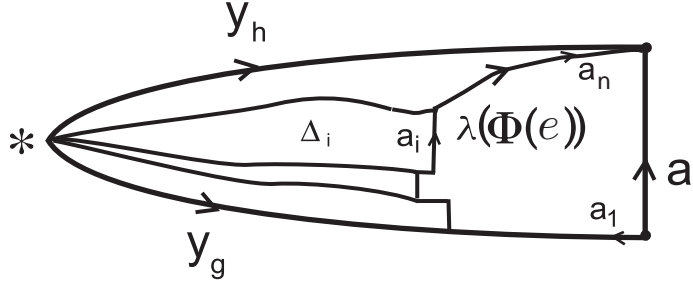}
\caption{Van Kampen diagram $\dd_e$ for recursive edge $e$}\label{fig:genericstacking}
\end{center}
\end{figure}

The penultimate step of the \srecp\  
is to eliminate repetitions in order
to obtain the \nff\ associated to
the flow function. 
Given any undirected edge $e$ in $E(X)$, choose 
$\dd_e$ to be a \edg\ constructed above for one of the
orientations of $e$.
Then the collection $\cc$ of normal forms,
together with the collection 
$\{\dd_e \mid e \in E(X)\}$
of {\edg}s, is a \nff\ for the \fstkbl\ group $G$.
Finally, we apply the seashell method again
to obtain a filling.


\begin{remark}\label{rmk:anfhasnf}
{\em
The \nff\ and filling constructed by the
\srecp\ satisfy another useful property:
}
For every van Kampen diagram $\dd$ in the filling and 
every vertex $v$ in $\dd$,
there is an edge path in $\dd$ from the basepoint $*$ to $v$
labeled by the normal form in $\cc$ for the element
$\pi_{\dd}(v)$ in $G$.  
\end{remark}

As with the earlier \sredp, we have an 
algorithm in the case that the set $\alg$ is computable. 

\begin{proposition}\label{prop:fillalg}
If $G$ is \afstkbl\ over the finite 
generating set $A$, then the \srecp\ is an inductive
algorithm which, upon input of a word $w \in A^*$
that represents the identity in $G$, will construct a van 
Kampen diagram for $w$ over the \stkg\ presentation.
\end{proposition}

Although the \sredp\ (from the proof of
Proposition~\ref{prop:prefixclosed}) for finding 
normal forms for a \stkbl\ group can also be used to describe the 
van Kampen diagrams in this filling, 
it is this inductive view that connects
 more directly to the algorithms for solving the
word problem and building van Kampen diagrams 
in the cases of almost convex groups and
groups with finite complete rewriting systems.


\begin{remark} {\em
For finitely generated groups that are not finitely
presented, the concept of stackability can still 
be defined, although in this case 
it makes sense to discuss \stkg\ maps in terms
of a presentation for $G$, to avoid the (somewhat
degenerate) case in
which every relator is included in the presentation.
A group $G$ with symmetrized presentation
$\pp=\langle A \mid R\rangle$ is } $\pp$-\stkbl\  
{\em if there is a bounded flow function $\ff$
for a maximal tree in $\ga(G,A)$
satisfying the condition that the \stkg\ relation
set $R_{\lbl \circ \ff}$ is a subset of $R$. 
Although we do not consider $\pp$-\stkbl\ groups
further in this paper, we note here
that the \sredp\ for finding normal forms and the
inductive method for constructing van Kampen diagrams
over the presentation $\langle A \mid R_{\lbl \circ \ff}\rangle$ 
of $G$ (and hence over $\pp$) described above
still hold in this more general setting.  
}
\end{remark}


\vspace{.1in}


\section{Examples of stackable groups}\label{sec:examples}



\subsection{Illustration: Solvable Baumslag-Solitar groups}\label{subsec:bs}


$~$

\vspace{.1in}

The solvable Baumslag-Solitar groups are presented by  
$G=BS(1,p)=\langle a,t \mid tat^{-1}=a^p \rangle$ with $p \in \Z$.
A set of normal forms over the generating set
$A=\{a,a^{-1},t,t^{-1}\}$ is given by
$$\cc:=\{t^{-i}a^mt^k \mid i,k \in \N \cup \{0\},~m \in \Z,
\text{ and either } p \not| m \text{ or } 0 \in \{i,k\}\}.$$
The recursive edges in $\ves$
are the directed edges in the Cayley graph $\ga(G,A)$
of the form $e_{w,b}$
with initial vertex labeled $w$ and edge label $b \in A$
satisfying one of the following:
\begin{enumerate}
\item $w=t^{-i}a^m$ and $b=t^\eta$ with $m \neq 0$, $\eta \in \{\pm 1\}$,
and $-i+\eta \le 0$, or
\item $w=t^{-i}a^mt^k$ and $b=a^\eta$ with $k>0$ and $\eta \in \{\pm 1\}$.
\end{enumerate}
%
We define a function $\sff:\ves \ra A^*$ by
$\sff(e_{t^{-i}a^m,t^{\eta}}):=(a^{-\nu p}ta^{\nu})^\eta$
in case (1), 
where $\nu:=\frac{m}{|m|}$ is 1 if $m>0$ and $-1$ if $m<0$, and
$\sff(e_{t^{-i}a^mt^k,a^{\eta}}):=t^{-1}a^{\eta p}t$
in case (2).  
Moreover let $\ff:\vece \ra \vec P$ be defined by
$\ff|_{\dgd} = id_{\dgd}$
and $\ff|_{\ves} = \path \circ (\init \times \sff)|_{\ves}$.

A portion of the Cayley graph and
the corresponding function $\ff$ in the case that $p=2$
are illustrated in Figure~\ref{fig:bs12};
here thickened edges are degenerate, and the images $\ff(e)$
for two recursive (dashed) edges $e$ are shown with widely dashed paths.
The double arrows within 2-cells indicate the direction of
the flow through the Cayley complex toward the basepoint
given by the function $\ff$.
\begin{figure}
\begin{center}
\includegraphics[width=2.6in,height=2.2in]{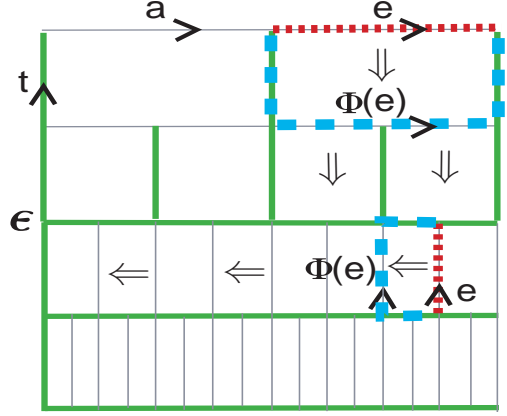}
\caption{Flow function for $BS(1,2)$}\label{fig:bs12}
\end{center}
\end{figure}

Properties (F1) and (F2d) of the definition of
a flow function, as well as boundedness, 
follow directly.  To show that $\ff$
also satisfies property (F2r), we first briefly describe the 
Cayley complex $X$ for the finite presentation above; see
for example \cite[Section~7.4]{echlpt} for more details.
The Cayley complex $X$ is homeomorphic to the product
$\R \times T$ of the real line with a regular tree $T$,
and there are projections 
$\Pi_\R:X \ra \R$ and $\Pi_T:X \ra T$.
The projection $\Pi_T$ takes each edge labeled by 
an $a^{\pm 1}$ to a vertex of $T$.  Each edge of $T$ is
the image of infinitely many $t$ edges of the 1-skeleton $X^1$, with
consistent orientation, and so we may consider the
edges of $T$ to be oriented and labeled by $t$, as well.
For the normal form $y_g=t^{-i}a^mt^k \in \cc$ of an element
$g \in G$, the projection onto $T$ of the path in $X^1$ starting at
$\ep$ and labeled by $y_g$ is the unique geodesic
path, labeled by $t^{-i}t^k$,
in the tree $T$ from $\Pi_T(\ep)$ to $\Pi_T(g)$. 
For any directed edge $e$ in 
$\ves$ in case (2) above, 
there are
$p+1$ 2-cells in the Cayley complex $X$ 
that contain $e$ in their boundary,
and the path $\ff(e)$ 
is the portion 
of the boundary, disjoint from $e$, 
of the only one of those 2-cells $\sigma$ that satisfies 
$d_T(\Pi_T(\ep),\Pi_T(q)) \le d_T(\Pi_T(\ep),\Pi_T(e))$
for all points $q \in \sigma$, where $d_T$
is the path metric in $T$.  For any edge $e'$ that
lies both in this $\ff(e)$ path and in $\ves$, then
$e'$ is again a recursive edge of type (2), and we
have $d_T(\Pi_T(\ep),\Pi_T(e'))<d_T(\Pi_T(\ep),\Pi_T(e))$.  
Thus the well-ordering on 
$\N$ applies, to show that the relation $<_\ff$ is
a well-founded strict partial ordering on the case (2)
edges in $\ves$.

The other projection map $\Pi_\R$ takes each vertex $t^{-i}a^mt^k$
to the real number $p^{-i}m$, and so takes each edge
labeled by $t^{\pm 1}$ to a single real number, and takes
each edge labeled $a^{\pm 1}$ to an interval in $\R$.
For an edge $e \in \ves$ in case (1) above, there are
exactly two 2-cells in $X$ containing $e$, and
the path $\ff(e)$ 
starting at the initial vertex $w=t^{-i}a^m$ of $e$
travels around the boundary of the one of these
two cells (except for the edge $e$)
whose image, under the projection $\Pi_\R$,
is closest to $0$.  The only possibly recursive edge 
$e'$ in the path $\ff(e)$
must also have type (1), and moreover
the initial vertex of $e'$ is $w'=t^{-i}a^{m-\nu}$ and
satisfies $|\Pi_\R(w')|=|\Pi_\R(w)|-p^{-i}$.
Then in all cases
the relation $<_\ff$ is a well-founded strict partial ordering,
completing the proof of property (F2r).
Therefore the function $\ff$ is a bounded flow function,
and the symmetrization of the
presentation above is the \stkg\ presentation.


\subsection{Groups admitting complete
     rewriting systems}\label{sec:rs}


$~$

\vspace{.1in}


A {\em finite complete rewriting system}
(finite {\em CRS}) for a group $G$ consists of a finite set $A$
and a finite set of ``rules'' $R \subseteq A^* \times A^*$
(with each $(u,v) \in R$ written $u \ra v$)
such that 
as a monoid, $G$ is presented by 
$G = Mon\langle A \mid u=v$ whenever $u \ra v \in R \rangle,$
and the rewritings
$xuy \ra xvy$ for all $x,y \in A^*$ and $u \ra v$ in $R$ satisfy:
(1) Each $g \in G$ is 
represented by exactly one {\em irreducible} word 
(i.e. word that cannot be rewritten)
over $A$, and
(2) the relation on $A^*$ defined by
$x>y$ whenever $x \rightarrow
x_1 \rightarrow ... \rightarrow x_n \rightarrow y$ is 
a well-founded strict partial ordering.
(That is, there is no infinite chain $w \ra x_1 \ra x_2 \ra \cdots$
of rewritings.)

Given any finite CRS $(A,R)$ for $G$, there is another
finite CRS $(A,R')$ for $G$ with the same
set of irreducible words such that the CRS is {\em minimal}.  
That is, for each $u \ra v$ in $R'$, the word $v$
and all proper subwords of the word $u$ 
are irreducible (see, for example, \cite[p.~56]{sims}).
If there is a letter $a \in A$ with $a=_G 1$, then
the rewriting $a \ra 1$ must be an element of $R'$,
and  for all other $u \ra v \in R'$, the letter $a$
cannot appear in the words $u$ or $v$.
Let $A'$ be the set $A$ with all letters representing
the identity of $G$ removed, and let $R''$ be the
set $R'$ with all rules of the form $a \ra 1$ for
$a \in A \setminus A'$ removed.  Now $(A',R'')$ is
also a minimal finite CRS for $G$ over $A$ with
the same set of irreducible words.
Next let $A''$ be the closure of $A'$ under inversion.
For each letter $a \in A'' \setminus A'$,
there is an irreducible word $z_a \in A^*$ with $a =_G z_a$.
Let $R''':=R'' \cup \{a \ra z_a \mid a \in A' \setminus A\}$.
Then  $(A'',R''')$ is again a minimal finite CRS for $G$,
and with the same set of irreducible normal forms
as the original CRS $(A,R)$.
For the remainder of this paper, we will assume that
all of our complete rewriting systems are minimal
and have an inverse-closed alphabet that does not contain
a representative~of~$\ep$.


Given any word $w \in A^*$, we write
$w \ras w'$ if there is any sequence of rewritings
$w=w_0 \ra w_1 \ra \cdots \ra w_n=w'$ (including
the possibility that $n=0$ and $w'=w$).
A {\em prefix rewriting}
of $w$ with respect to the complete rewriting system $(A,R)$
is a sequence of rewritings $w=w_0 \ra \cdots \ra w_n=w'$,
written $w \prs w'$,
such that at each $w_i$,
the shortest possible reducible prefix is rewritten to obtain
$w_{i+1}$.  
When $w_n$ is irreducible,
the number $n$ is the {\em prefix rewriting length} of $w$,
denoted $prl(w)$.

In Theorem~\ref{thm:crsrecit}, we 
apply ideas developed  in the 
construction of a 1-combing
associated to a finite complete rewriting
system by Hermiller and Meier 
in~\cite{hmeiertcacrs}, in order
to build a stackable structure from a finite CRS.   

\begin{theorem}\label{thm:crsrecit}
A group admitting a
finite complete rewriting system
is \astkbl.
\end{theorem}

\begin{proof}
Let $\cc=\{y_g \mid g \in G\}$ be the set of irreducible words from 
a minimal
finite CRS $(A,R)$ for a group $G$.
Then $\cc=A^* \setminus \cup_{u \ra v \in R}A^*uA^*$.
Note that prefixes of irreducible
words are also irreducible, and so $\cc$ is a 
prefix-closed set of normal 
forms for $G$ over $A$.

Let $\ga$ be the Cayley graph for the pair $(G,A)$.
As usual, the directed edge $\ega$ in $\ga$
with label $a$
and initial vertex $g$ lies in the set
$\dgd$ of degenerate edges if and only if
$y_gay_{ga}^{-1}$ freely reduces to the empty word,
which in turn holds if and only
if the undirected edge underlying $\ega$ lies 
in the tree of edges in paths from $\ep$ labeled by
words~in~$\cc$.


Define a function $\ff:\vece \ra \vec P$ as follows.
On degenerate edges, $\ff|_{\dgd} := id_{\dgd}$. 
Given a recursive edge $\ega \in \ves$,  the word
$y_ga$ is reducible, and
since $y_g$ is
irreducible, the shortest reducible prefix of
$y_ga$ is the entire word.
Minimality of the rewriting system $R$ implies that
there is a unique factorization $y_g=w\tilde u$ 
such that $\tilde ua$ is the left
hand side of a unique rule $\tilde ua \ra v$ in $R$; that is,
$y_ga \ra wv$ is a prefix rewriting.
Then define $\ff(\ega):=\path(g,\tilde u^{-1}v)$.

Properties (F1) and (F2d) of the definition of flow function are 
immediate.  To check property (F2r), 
we first let $p$ be the path $\ff(\ega)$ above
labeled $\lbl(\ff(\ega))=\tilde u^{-1}v$
for $\ega \in \ves$.
Since the word $\tilde u$ is a suffix of the normal form $y_g$,
then the edges in the path $p$
that correspond to the letters in 
$\tilde u^{-1}$ all lie in the set $\dgd$ of degenerate edges.  
For each directed edge $e'$ in the subpath of $p$ labeled by $v$,
either $e'$ also lies in $\dgd$, or else $e' \in \ves$ and
there is a factorization $v=v_1a'v_2$ so that
$e'$ is the directed edge along $p$ corresponding 
to the label $a' \in A$.
In the latter case, if we denote
the initial vertex of $e'$ by $g'$, then
the prefix rewriting sequence
from $y_{g'}a'v_2$ to its irreducible form
is a (proper) subsequence of the prefix
rewriting of $y_ga$.  
That is, if we define a function
$prl:\ves \ra \N$ by $prl(e_{h,b}):=prl(y_hb)$
whenever $e_{h,b} \in \ves$,
we have $prl(e')<prl(\ega)$.
Hence the ordering $<_{\ff}$ corresponding to our
function $\ff:\vece \rightarrow \vec P$ satisfies 
the property that $e'<_{\ff} e$ implies $prl(e')<prl(e)$,
and the well-ordering property on $\N$ implies that
$<_{\ff}$ is a well-founded strict partial ordering.
Thus (F2r) holds as well, and $\ff$ is a flow function.

The image set $\lbl(\ff(\vece))$ is the set
of words $A \cup \{\tilde u^{-1}v \mid \exists a \in A$
with $\tilde u a \ra v$ in $R\}$.  Thus 
boundedness of $\ff$
follows from finiteness of the sets $A$ and $R$
of generators and rules in the rewriting system.
Then Corollary~\ref{cor:stkbldefs} shows that
$G$ is \stkbl\ over the generating set $A$.

To determine whether a tuple $(w,a,x)$ (where $w,x \in A^*$
and $a \in A$) lies in the 
associated set 
$\alg$, we begin by computing the normal forms
$y_w$ and $y_{wa}$ from $w$ and $wa$, using the
rewriting rules of our finite system.
Then $(w,a,x) \in \alg$ if and only if
either at least one of the words 
$y_wa$ and $y_{wa}a^{-1}$ is irreducible 
and $a=x$, or
else both of the words $y_wa$ and $y_{wa}a^{-1}$ are 
reducible and there exist a factorization
$y_w=z\tilde u$ for some
$z \in A^*$ and a rule $\tilde ua \ra v$ in $R$ such that
$x=\tilde u^{-1}v$.  Since there are only
finite many rules in $R$ to check for such a
decomposition of $y_w$,
it follows that the set $\alg$ is 
also computable, and
so this stackable structure  is algorithmic.
\end{proof}



\subsection{Thompson's group $F$}\label{subsec:f}


$~$

\vspace{.1in}

Thompson's group 
\[F=\langle x_0,x_1 \mid [x_0x_1^{-1},x_0^{-1}x_1x_0],
   [x_0x_1^{-1},x_0^{-2}x_1x_0^2] \rangle\]
is the group of orientation-preserving piecewise linear
homeomorphisms of the unit interval [0,1], satisfying that
each linear piece has a slope of the form $2^i$ for some
$i \in \Z$, and all breakpoints occur in the 2-adics.
In \cite{chst}, Cleary, Hermiller, Stein, and Taback
effectively show that Thompson's group with the generating set
$A=\{x_0^{\pm 1},x_1^{\pm 1}\}$ is \fstkbl, 
with \stkg\ presentation given by 
the symmetrization of the presentation above. 
Moreover, in~\cite[Definition~4.3]{chst} they give an 
algorithm for computing the \stkg\ map, which 
can be used to show that $F$ is \astkbl.

Although we will not repeat their proof here, 
we describe 
the normal form set $\cc$ associated to the \stkbl\ structure 
constructed for Thompson's group in~\cite{chst}
in order to
discuss its formal language theoretic properties.
Given a word $w$ over the generating set 
$A=\{x_0^{\pm 1},x_1^{\pm 1}\}$, denote the number of occurrences in $w$ of the letter
$x_0$ minus the number of occurrences in $w$ of the letter $x_0^{-1}$ by $expsum_{x_0}(w)$;
that is, the exponent sum for $x_0$.
The authors of that paper show (\cite[Observation~3.6(1)]{chst}) that the set 

\smallskip

$\cc:=\{w \in A^* \mid$ for all $\eta \in \{\pm 1\}$, the words
$x_0^\eta x_0^{-\eta}$, $ x_1^\eta x_1^{-\eta}$, and $x_0^2x_1^\eta$ 
are  not 

\hspace{1in} subwords of $w$, and for all prefixes $ w' $ of $ w,
expsum_{x_0}(w') \le 0\}$,

\smallskip


\noindent is a set of normal forms for $F$.  
Moreover, each of these words
labels a (6,0)-quasi-geodesic path in the
Cayley complex $X$~\cite[Theorem~3.7]{chst}.

This set $\cc$ is the intersection of the regular language
$A^* \setminus \cup_{u \in U}A^*uA^*$, where 
$U:=\{x_0x_0^{-1},x_0^{-1}x_0,x_1x_1^{-1},x_1^{-1}x_1,x_0^2x_1,x_0^2x_1^{-1}\}$,
with the language $L:=\{w \in A^* \mid$ for all prefixes $ w' $ of $ w,
expsum_{x_0}(w') \le 0\}$.  
We refer the reader to the text of Hopcroft and Ullman~\cite{hu}
for definitions and results on context-free and regular languages we now
use to analyze the set $L$.
The language $L$ can be recognized by a deterministic push-down automaton (PDA)
which pushes an $x_0^{-1}$ onto its stack whenever
an $x_0^{-1}$ is read, and pops an $x_0^{-1}$ off of its stack whenever
an $x_0$ is read.  When $x_1^{\pm 1}$ is read, the PDA
does nothing to the stack, and does not change its state.
The PDA remains in its initial state unless an $x_0$ is read 
when the only symbol on 
the stack is the stack start symbol $Z_0$, in which case
the PDA transitions to a fail state (at which it must then
remain upon reading the remainder of the input word).  
Ultimately the PDA accepts
a word whenever its final state is its initial state.
Consequently, $L$ is a deterministic context-free language.  
Since the intersection of a regular language with a deterministic
context-free
language is deterministic context-free, 
the set $\cc$ is also a deterministic context-free language.

In Section~\ref{sec:rs}, the normal form set 
of the stackable structure for a group with a finite
complete rewriting system, 
$\cc=A^* \setminus \cup_{u \ra v \in R}A^*uA^*$,
is a regular language.  In consideration of the
open question~\cite{thompsonpbms} of whether
or not Thompson's group $F$ has a finite complete
rewriting system, it would also be of interest
to know whether or not Thompson's group $F$
is \stkbl\ with respect to a regular language
of normal forms.


\subsection{Almost convex groups}\label{subsec:ac}


$~$

\vspace{.1in}

Let $G$ be a group with an
inverse-closed generating set $A$, and
let $d=d_\ga$ be the path metric on the associated
Cayley graph $\ga$.  For $n \in \N$, 
define the sphere $S(n)$ of radius $n$
to be the set of points in $\ga$ a distance
exactly $n$ from the vertex labeled by the
identity $\ep$, and define the ball $B(n)$ of radius $n$
to be the set of points in $\ga$ whose path metric distance
to $\ep$ is less than or equal to $n$.

\begin{definition}~\cite{cannon}\label{def:ac}
A group $G$ is {\em almost convex} with respect
to the finite symmetric
generating set $A$ if there is a constant $k$
such that
for all $n \in \N$ and for all $g,h$ in the
sphere $S(n)$ satisfying
$d_\ga(g,h) \leq 2$
(in the Cayley
graph $\ga=\ga(G,A)$), there is a path inside 
the ball $B(n)$ from
$g$ to $h$ of length no more than $k$.
\end{definition}

Cannon~\cite{cannon} showed that every group 
satisfying an almost convexity condition
over a finite generating set is also finitely presented.
Thiel~\cite{thiel} showed that almost convexity
is a property that depends upon the finite
generating set used.  

In the proof of
Theorem~\ref{thm:aceti} below, 
we show that 
a pair $(G,A)$ that is almost convex 
is \afstkbl.  Moreover the class of almost convex groups is
exactly the class of geodesically stackable groups, and 
this must hold with respect to the shortlex normal forms.
Given a choice of total ordering on $A$, 
a word $z_g \in A^*$ is the {\em shortlex normal form}
for $g \in G$ if $\rep(z_g)=g$ and whenever $w \in A^*$
with $w=_G z_g$, then either the word lengths (in $A^*$) satisfy
$l(w)>l(z_g)$, or else $l(w)=l(z_g)$ and $w$ is lexicographically
greater than $z_g$ with respect to the ordering on $A$.

\begin{theorem}\label{thm:aceti}
Let $G$ be a group with finite generating set $A$.  The following
are equivalent:
\begin{enumerate}
\item The pair $(G,A)$ is almost convex.
\item The pair $(G,A)$ is geodesically \stkbl.
\item The pair $(G,A)$ is geodesically \afstkbl\ with 
respect to shortlex normal forms.
\end{enumerate}
\end{theorem}

\begin{proof}
Suppose that the group $G$ has a finite symmetric
generating set $A$, and let $\ga$ be the corresponding
Cayley graph with path metric $\td$.   The implication
(3) $\Rightarrow$ (2) is immediate.

\smallskip

\noindent{\em (1) implies (3):}

Suppose that the group $G$ 
is almost convex with respect to $A$, 
with an almost convexity constant $k$.
Let $\cc=\{z_g \mid g \in G\}$ be the 
set of shortlex  normal forms over $A$ for $G$
(with respect to any choice of total ordering of $A$). 
Define a relation $<_\alpha$ on the set $\ves=\vec E_{r,\cc}$ of 
recursive edges by $e'<_\alpha e$ whenever $\alpha(e')<\alpha(e)$,
where $\alpha(\ega):=\frac{1}{2}(d(\ep,g)+d(\ep,ga))$
for all $\ega \in \ves$; then $<_\alpha$ inherits
the property of being a well-founded
strict partial ordering from the usual ordering on 
$\N[\frac{1}{2}]$.

Define a function $\sff:\ves \ra A^*$ as follows.
Let $\ega$ be any element of $\ves$. 

\noindent {\em Case I.} If $\td(\ep,g)=\td(\ep,ga)=n$, then 
the points $g$ and
$ga$ lie in the same sphere, and almost convexity of
$(G,A)$ 
implies that there is a directed edge path in $\ga$ from $g$ to $ga$
of length at most $k$ 
that lies in the ball $B(n)$.  
In this case define $\sff(\ega)$ to be the shortlex 
least word over $A$
that labels a path in $B(n)$ from $g$ to $ga$.
For any edge $e' \in \ves$ lying in the path $\path(g,\sff(\ega))$
(starting at $g$ and labeled by $\sff(\ega)$),
the midpoint $p$ of $e'$ lies in $B(n)$,
and so at least one of the endpoints of $e'$ must
lie in $B(n-1)$.  Then $\alpha(e') \le n-\frac{1}{2}<n=\alpha(\ega)$.

\noindent {\em Case II.}
If $\td(\ep,g)=n$ and $\td(\ep,ga)=n+1$,
then 
we can write $z_{ga} =_{A^*} z_{h}b$ for some $h \in G$
and $b \in A$.
Hence $g,h \in S(n)$ and $d_\ga(g,h) \le 2$.  
In this case we 
define $\sff(\ega):=xb$ where $x$ is the shortlex least
word over $A$ that labels 
a path in $B(n)$
from $g$ to $h$.  The path
$\path(g,\sff(\ega))$ has length at most
$k+1$, and the final edge
in this path, labeled by $b$, is degenerate.  Thus any recursive
edge $e' \in \ves$ in this path lies
in $B(n)$, and we have
$\alpha(e') \le n-\frac{1}{2}<n+\frac{1}{2}=\alpha(\ega)$
in this case.

\noindent {\em Case III.}
If $\td(\ep,g)=n+1$ and $\td(\ep,ga)=n$, then
$z_g=_{A^*} z_{g'}c$ for some $c \in A$ and $g' \in G$, and we define
$\sff(\ega):=c^{-1}y$ where $y$ is
the shortlex least word, of length at most $k$, labeling a path
in $B(n)$ from $g'$ to $ga$.
The initial edge of $\path(g,\sff(\ega))$ labeled $c^{-1}$
is degenerate, and as in Case II we have $\alpha(e')<\alpha(\ega)$
for all recursive edges $e'$ in this path.


In all cases there is a path of length at most $k+1$
satisfying the conditions of Definition~\ref{def:fstkbl}
for the ordering $<_\alpha$ on $\ves$, and $\sff$ is a stacking map.
Since the words in $\cc$ are geodesics,
then $G$ is geodesically \stkbl\ over $A$
with respect to the shortlex normal forms.
  
We are left with showing computability
for the subset 
$$
\alg=\{(w,a,\sff(e_{w,a})) \mid w \in A^*,~a \in A,~\wa \in \ves\}
\cup \{(w,a,a) \mid w \in A^*,~a \in A,~\wa \in \dgd\}
$$ 
of $A^* \times A \times A^*$.
Suppose that $(w,a,x)$ is any element of 
$A^* \times A \times A^*$.
Cannon~\cite[Theorem~1.4]{cannon} has shown
that the word problem is solvable for $G$, and
so by enumerating the words in $A^*$ in increasing
shortlex order, and checking whether each in turn
is equal in $G$ to $w$, we can find the shortlex
normal form $z_w$ for $w$.  Similarly we compute $z_{wa}$.  
If the word $z_w a z_{wa}^{-1}$ freely reduces
to $1$, then the tuple $(w,a,x)$ lies in $\alg$ 
if and only if $x=a$.

Suppose on the other hand that the word
$z_w a z_{wa}^{-1}$ does
not freely reduce to $1$.  
If (as in Case I above)
the word lengths
$l(z_w)=l(z_{wa})$ both equal a natural number $n$, then
we enumerate the
elements of the finite set $\cup_{i=0}^k A^i$
of words of length up to $k$ in increasing
shortlex order.  For each word $y=a_1 \cdots a_m$
in this enumeration, with each $a_i \in A$,
we use the word problem solution again 
to compute the word length $l_{y,i}$
of the normal form $z_{wa_1 \cdots a_i}$ for
each $0 \le i \le m$.
If each $l_{y,i} \le n$, and equalities $l_{y,i} = n$
do not hold for two consecutive indices $i$, then
$(w,a,x)$ lies in $\alg$ if and only if $x=y$
and we halt the enumeration; otherwise, we go on to 
check the next word in our enumeration.
The argument for Cases II-III in which 
$l(z_w)=l(z_{wa}) \pm 1$ are
similar.  

Combining the algorithms in the previous two paragraphs,
the set $\alg$ is computable and the stackable
structure above for $G$ is algorithmic.

\noindent {\em (2) implies (1):}

Suppose that the group $G$ is geodesically stackable
over the generating set $A$
 with respect to a set $\cc$ of (geodesic) normal forms,
and let $\ttt$ be the corresponding tree of degenerate edges. 
Let $\sff:\ves=\vec E_{r,\cc} \ra A^*$ be
an associated stacking map and let $\ff:\vece \ra \vec P$
be the corresponding bounded flow 
function from Corollary~\ref{cor:stkbldefs}.
Let $M:=\max\{l(\sff(e)) \mid e \in \ves\}$ and let $k:=2M^2+2$.
Also let $g,h$ be any two points in a sphere
$S(n)$ with $d_\ga(g,h) \le 2$.

If $d(g,h) = 1$, then $h=_G ga$ for some $a \in A$.
Moreover, since all normal forms in $\cc$ are geodesics,
the edge $\ega$ from $g$ to $h$ must be recursive.  Then
the path $p:=\ff(\ega)$ labeled $\sff(\ega)$ of length $\le M<k$ 
from $g$ to $h$ satisfies the property that for every edge $e'$ in 
$p$, either $e' \in \dgd$ or
else $e' \in \ves$ with $e' <_{\ff} e$.  
Whenever $e' \in \ves$, then
applying Definition~\ref{def:geostk} we have
$\alpha(e')<\alpha(\ega) =n$, and so the edge $e'$ 
must lie in $B(n)$.
If needed we replace each subpath of $p$
whose edges all lie in $\dgd$ by the shortest
path in the tree $\ttt$ of degenerate edges
between the same endpoints.  The effect of this
replacement can only shorten the path $p$,
and all edges in the new path must
lie in $B(n)$.

On the other hand, suppose that $d(g,h)=2$,
with $h=gab$ for some $a,b \in A$.
If $d(\ep,ga)=n-1$, then there is a path
of length $2 \le k$ from $g$ to $h$ lying inside $B(n)$,
and if $d(\ep,ga)=n$, we can apply the previous
paragraph twice to obtain a path of length at most
$2M<k$ from $g$ to $h$ via $ga$.
Finally consider the case that $d(\ep,ga)=n+1$,
and write the (geodesic) normal form for $ga$ 
as $y_{ga}=y_{g'}a'$ where $g' \in S(n)$
and $a' \in A$.  It suffices to show that
there is a path in $B(n)$ from $g$ to $g'$
of length at most $M^2<\frac{1}{2}k$, since a similar
proof results in such a path from $g'$ to $h$.
If $g=g'$ we are done, so suppose that $g \neq g'$.

Now the 
edge $e_{g',a'}$ lies in the tree $\ttt$
defined by $\cc$.  Since the normal forms are geodesic,
there can only be one directed edge in $\ttt$
ending at $ga=g'a'$ and starting at a point in $S(n)$, and
so the edge $\ega$ from $g$ to $ga$ must be recursive.
Each recursive edge $e'$ in the 
path $p:=\ff(\ega)$
satisfies $\alpha(e')<\alpha(\ega)=n+\frac{1}{2}$,
and so both endpoints of $e'$ lie in $B(n)$.
Replace each recursive edge $e'$ in the path $p$
satisfying $\alpha(e')=n$ by
the directed path $\ff(e')$, to obtain a new 
directed path $p'$
of length at most $M^2$ from $g$ to $ga$.
Now for every recursive edge $e''$ in the path $p'$,
we have $\alpha(e'')<n$, and so
all of the recursive edges in the path $p'$
lie in $B(n)$.  

Next as above we replace
each subpath of $p'$ consisting
solely of degenerate edges in $\ga$  by the shortest
path in the tree $\ttt$ 
between the same endpoints, resulting in another path
$p''$ from $g$ to $ga$ all of whose recursive
edges lie in $B(n)$.  The path $p''$ must
end with a path in the tree $T$ from a point in
$B(n)$ to the vertex $ga$, and therefore the
last directed edge of this path is the edge $e_{g',a'}$.
Let $\tilde p$ be the path $p''$ with this
last edge removed.  Then $\tilde p$
is a path from $g$ to $g'$
lying in $B(n)$ of length at most
$M^2$, as required.
\end{proof}

\begin{remark} {\em 
Cannon's word 
problem algorithm for almost convex groups, which we
applied in the proof of Theorem~\ref{thm:aceti},
requires the use of an enumeration of a
finite set of words over $A$, namely
those that represent $\ep$ in $G$ and have length at most 
$k+2$.  As Cannon
also points out~\cite[p.~199]{cannon}, although this
set is indeed recursive, there may not be an algorithm
to find this set, starting from $(G,A)$ and the constant $k$.}
\end{remark}




\end{document}